\documentclass[a4paper]{scrartcl}

\usepackage[utf8]{inputenc}
\usepackage[T1]{fontenc}

\usepackage{amsmath, amsthm, amssymb}
\usepackage{mathtools}
\usepackage{mathrsfs}
\usepackage{bm}
\usepackage{url}
\usepackage{enumitem}
\newlist{hypothenum}{enumerate}{3}
\setlist[hypothenum,1]{label=(\roman*)}
\usepackage[all,cmtip]{xy}
\usepackage{tikz-cd}
\usepackage{tikz}
\usetikzlibrary{calc}
\usetikzlibrary{patterns}
\usepackage{relsize}
\usepackage{tabularx}
\newcolumntype{C}{>{\centering\arraybackslash}X} 
\usepackage{multirow}
\usepackage{arydshln}
\usepackage[referable]{threeparttablex}
\renewlist{tablenotes}{enumerate}{1}
\makeatletter
\setlist[tablenotes]{label=\tnote{\alph*},ref=\alph*,itemsep=\z@,topsep=\z@skip,partopsep=\z@skip,parsep=\z@,itemindent=\z@,labelindent=\tabcolsep,labelsep=.2em,leftmargin=*,align=left,before={\footnotesize}}
\makeatother
\usepackage{longtable}
\usepackage{needspace}
\usepackage{caption}
\usepackage{subcaption}

\usepackage{afterpage}

\numberwithin{equation}{section}

\theoremstyle{plain}
\newtheorem*{mainthm}{Main Theorem}
\newtheorem{theorem}{Theorem}[section]
\newtheorem*{problem*}{Problem}
\newtheorem{proposition}[theorem]{Proposition}

\newtheorem{lemma}[theorem]{Lemma}
\newtheorem{corollary}[theorem]{Corollary}

\newtheorem{assumption}[theorem]{Assumption}

\theoremstyle{definition}
\newtheorem{definition}[theorem]{Definition}

\theoremstyle{remark}
\newtheorem{remark}[theorem]{Remark}

\newcommand{\eps}{\varepsilon}
\newcommand{\setsuch}[2]{\left\{ #1 \; \middle| \; #2 \right\}}

\newcommand{\subl}{{\mathsmaller{<}}}
\newcommand{\subg}{{\mathsmaller{>}}}
\newcommand{\sube}{{\mathsmaller{=}}}

\newcommand{\subge}{{\mathsmaller{\geq}}}
\newcommand{\subenz}{\Bumpeq}

\newcommand{\sublap}{{\mathsmaller{\lesssim}}}
\newcommand{\subgap}{{\mathsmaller{\gtrsim}}}
\newcommand{\transl}{t}

\newcommand{\extra}{a}

\DeclareMathOperator{\Span}{Span}

\DeclareMathOperator{\Aff}{Aff}

\DeclareMathOperator{\Ad}{Ad}

\DeclareMathOperator{\GL}{GL}

\DeclareMathOperator{\PSL}{PSL}
\DeclareMathOperator{\SO}{SO}
\DeclareMathOperator{\PSO}{PSO}

\DeclareMathOperator{\Orth}{O}
\newcommand{\fundef}[5]{
\entrymodifiers={+!!<0pt,\fontdimen22\textfont2>}
\xymatrix@R=3pt{\llap{$#1$\;\;} {#2} \ar@{->}[r] & {#3} \\ {#4} \ar@{|->}[r] & {#5}}
} 

\newcommand{\ie}{i.e.\ }
\newcommand{\eg}{e.g.\ }

\newcommand{\jordan}{\operatorname{Jd}}

\begin{document}

\title{Construction of Milnorian representations}
\author{Ilia Smilga}
\date{}
              
\maketitle

\begin{abstract}
We prove a partial converse to the main theorem of the author's previous paper \emph{Proper affine actions: a sufficient criterion} (submitted; available at~\url{arXiv:1612.08942}). More precisely, let $G$ be a semisimple real Lie group with a representation $\rho$ on a finite-dimensional real vector space $V$, that does \emph{not} satisfy the criterion from the previous paper. Assuming that $\rho$ is irreducible and under some additional assumptions 
on~$G$ and~$\rho$, we then prove that there does \emph{not} exist a group of affine transformations acting properly discontinuously on~$V$ whose linear part is Zariski-dense in $\rho(G)$.
\end{abstract}

\section{Introduction}

\subsection{Background and motivation}
\label{sec:background}

The present paper is part of a larger effort to understand discrete groups $\Gamma$ of affine transformations (subgroups of the affine group $\GL_n(\mathbb{R}) \ltimes \mathbb{R}^n$) acting properly discontinuously on the affine space $\mathbb{R}^n$. The case where $\Gamma$ consists of isometries (in other words, $\Gamma \subset \Orth_n(\mathbb{R}) \ltimes \mathbb{R}^n$) is well-understood: a classical theorem by Bieberbach says that such a group always has an abelian subgroup of finite index.

We say that a group $G$ acts \emph{properly discontinuously} on a topological space $X$ if for every compact $K \subset X$, the set $\setsuch{g \in G}{g K \cap K \neq \emptyset}$ is finite. We define a \emph{crystallographic} group to be a discrete group $\Gamma \subset \GL_n(\mathbb{R}) \ltimes \mathbb{R}^n$ acting properly discontinuously and such that the quotient space $\mathbb{R}^n / \Gamma$ is compact. In \cite{Aus64}, Auslander conjectured that any crystallographic group is virtually solvable, that is, contains a solvable subgroup of finite index. Later, Milnor \cite{Mil77} asked whether this statement is actually true for any affine group acting properly discontinuously. The answer turned out to be negative: Margulis \cite{Mar83, Mar87} found a counterexample in dimension~$3$. On the other hand, Fried and Goldman \cite{FG83} proved that the Auslander conjecture does hold in dimension~$3$ (the cases $n=1$ and $2$ are easy). Recently, Abels, Margulis and Soifer \cite{AMS12} proved it in dimension $n \leq 6$, and independently Tomanov \cite{Tom16} proved it in dimension $n \leq 5$. See \cite{AbSur} for a survey of already known results.

Following Margulis's breatkthrough, numerous other counterexamples to Milnor's conjecture have been found. It is then natural to try to classify them by their Zariski-closure; in other words, we ask ourselves the following question:
\begin{problem*}
\label{proper_affine_existence}
Let $G$ be any real algebraic group, $\rho: G \to \GL(V)$ some representation on a finite-dimensional real vector space $V$. The data of $\rho$ then allows us to define the ``affine group'' $G \ltimes V$. Can we find a subgroup $\Gamma \subset G \ltimes V$ whose linear part is Zariski-dense in $G$ and that is free, nonabelian and acts properly discontinuously on the affine space corresponding to $V$?
\end{problem*}
Say that the representation~$\rho$ is \emph{non-Milnorian} if the answer is positive, \emph{Milnorian} if the answer is negative. Let us focus more specifically on the case where the group~$G$ is semisimple, and the representation~$\rho$ is irreducible.

In this setting, the author has found a sufficient condition~\cite{Smi16b} for a representation to be non-Milnorian (encompassing all the previously known examples of non-Milnorian representations, such as \cite{Mar83}, \cite{AMS02} and \cite{Smi14}; see the introduction to~\cite{Smi16b} for a brief summary of all of these partial results). He has conjectured that this condition is actually necessary and sufficient. This paper proves that, assuming that it satisfies some additional conditions, any representation that fails the test from~\cite{Smi16b} is indeed Milnorian.

So far, the following examples of Milnorian representations have been known:
\begin{itemize}
\item All the representations that do not have $0$ as a restricted weight are certainly Milnorian. This includes in particular the standard representation of $G = \SO^+(n, n)$ (acting on $V = \mathbb{R}^{2n}$). In Section~\ref{sec:non_radical} we have spelled out the proof of this fact, but this argument is obvious and has been known for a long time.
\item Abels, Marguls and Soifer \cite{AMS02} have proved that for all $n \geq 1$, the standard representation of $G = \SO^+(n+1, n)$ (acting on $V = \mathbb{R}^{2n+1}$) is Milnorian if $n$~is even (Theorem~A), non-Milnorian otherwise (Theorem~B).
\item They later proved (Theorem A in~\cite{AMS11}) that for all natural integers $p, q$ such that $|p - q| \geq 2$, the standard representation of $G = \SO^+(p, q)$ (acting on $V = \mathbb{R}^{p+q}$) is Milnorian.
\end{itemize}
The proof presented here allows us to derive Theorem~A from~\cite{AMS02} as a particular case, but does not yet cover the representations covered by Theorem~A from~\cite{AMS11}.

In order to state our theorem, we need to introduce a few classical notations.

\subsection{Basic notations}
\label{sec:lie}

For the remainder of the paper, we fix a semisimple real Lie group~$G$; let $\mathfrak{g}$~be its Lie algebra. Let us introduce a few classical objects related to~$\mathfrak{g}$ and~$G$ (defined for instance in Knapp's book \cite{Kna96}, though our terminology and notation differ slightly from his).

We choose in $\mathfrak{g}$:
\begin{itemize}
\item a Cartan involution $\theta$. Then we have the corresponding Cartan decomposition $\mathfrak{g} = \mathfrak{k} \oplus \mathfrak{q}$, where we call $\mathfrak{k}$ the space of fixed points of $\theta$ and $\mathfrak{q}$ the space of fixed points of $-\theta$. We call $K$ the maximal compact subgroup with Lie algebra $\mathfrak{k}$.
\item a \emph{Cartan subspace} $\mathfrak{a}$ compatible with $\theta$ (that is, a maximal abelian subalgebra of $\mathfrak{g}$ among those contained in $\mathfrak{q}$). We set $A := \exp \mathfrak{a}$.
\item a system $\Sigma^+$ of positive restricted roots in $\mathfrak{a}^*$. Recall that a \emph{restricted root} is a nonzero element $\alpha \in \mathfrak{a}^*$ such that the restricted root space
\[\mathfrak{g}^\alpha := \setsuch{Y \in \mathfrak{g}}{\forall X \in \mathfrak{a},\; [X, Y] = \alpha(X)Y}\]
is nontrivial. They form a root system $\Sigma$; a system of positive roots $\Sigma^+$ is a subset of $\Sigma$ contained in a half-space and such that $\Sigma = \Sigma^+ \sqcup -\Sigma^+$. (Note that in contrast to the situation with ordinary roots, the root system $\Sigma$ need not be reduced; so in addition to the usual types, it can also be of type~$BC_n$.)

We call~$\Pi$ be the set of simple restricted roots in~$\Sigma^+$. We call
\[\mathfrak{a}^{++} := \setsuch{X \in \mathfrak{a}}{\forall \alpha \in \Sigma^+,\; \alpha(X) > 0}\]
the (open) dominant Weyl chamber of $\mathfrak{a}$ corresponding to~$\Sigma^+$, and
\[\mathfrak{a}^{+} := \setsuch{X \in \mathfrak{a}}{\forall \alpha \in \Sigma^+,\; \alpha(X) \geq 0} = \overline{\mathfrak{a}^{++}}\]
the closed dominant Weyl chamber.
\end{itemize}
Then we call:
\begin{itemize}
\item $M$ the centralizer of $\mathfrak{a}$ in $K$, $\mathfrak{m}$ its Lie algebra.
\item $L$ the centralizer of $\mathfrak{a}$ in $G$, $\mathfrak{l}$ its Lie algebra. It is clear that $\mathfrak{l} = \mathfrak{a} \oplus \mathfrak{m}$, and well known (see \eg \cite{Kna96}, Proposition 7.82a) that $L = MA$.
\item $\mathfrak{n}^+$ (resp. $\mathfrak{n}^-$) the sum of the restricted root spaces of $\Sigma^+$ (resp. of $-\Sigma^+$), and $N^+ := \exp(\mathfrak{n}^+)$ and $N^- := \exp(\mathfrak{n}^-)$ the corresponding Lie groups.
\item $\mathfrak{p}^+ := \mathfrak{l} \oplus \mathfrak{n}^+$ and $\mathfrak{p}^- := \mathfrak{l} \oplus \mathfrak{n}^-$ the corresponding minimal parabolic subalgebras, $P^+ := LN^+$ and $P^- := LN^-$ the corresponding minimal parabolic subgroups.
\item $W := N_G(A)/Z_G(A)$ the restricted Weyl group.
\item $w_0$ the \emph{longest element} of the Weyl group, that is, the unique element such that $w_0(\Sigma^+) = \Sigma^-$. It is clearly an involution.
\end{itemize}

See Examples 2.3 and 2.4 in the author's earlier paper~\cite{Smi14} for working through these definitions in the cases $G = \PSL_n(\mathbb{R})$ and~$G = \PSO^+(n, 1)$.

Finally, if $\rho$~is a representation of~$G$ on a finite-dimensional real vector space~$V$, we call:
\begin{itemize}
\item the \emph{restricted weight space} in~$V$ corresponding to a form~$\lambda \in \mathfrak{a}^*$ the space
\[V^\lambda := \setsuch{v \in V}{\forall X \in \mathfrak{a},\; X \cdot v = \lambda(X)v};\]
\item a \emph{restricted weight} of the representation~$\rho$ any form~$\lambda \in \mathfrak{a}^*$ such that the corresponding weight space is nonzero.
\end{itemize}

\subsection{Statement of main result}

Let $\rho$~be an irreducible representation of~$G$ on a finite-dimensional real vector space~$V$. Without loss of generality, we may assume that $G$~is connected and acts faithfully. We may then identify the abstract group~$G$ with the linear group~$\rho(G) \subset \GL(V)$. Let~$V_{\Aff}$~be the affine space corresponding to~$V$. The group of affine transformations of~$V_{\Aff}$ whose linear part lies in~$G$ may then be written~$G \ltimes V$ (where $V$~stands for the group of translations). Here is the main result of this paper.
\begin{mainthm}
Suppose that $\rho$ satisfies the following conditions:
\begin{hypothenum}
\item \label{itm:bad} For every vector~$v$ fixed by all elements of~$L$, we have $\tilde{w}_0(v) = v$, where $\tilde{w}_0$ is any representative in~$G$ of~$w_0 \in N_G(A)/Z_G(A)$.
\item \label{itm:split} Every vector $v \in V$ that is fixed by all elements of~$A$ is actually fixed by all elements of~$L$;
\item \label{itm:non_swinging} There exists an element~$X_0 \in \mathfrak{a}$ such that~$-w_0(X_0) = X_0$ and for every nonzero restricted weight~$\lambda$ of~$\rho$, we have $\lambda(X_0) \neq 0$.
\end{hypothenum}
Then $\rho$~is Milnorian, \ie there does not exist a subgroup~$\Gamma$ in the affine group~$G \ltimes V$ whose linear part is Zariski-dense in~$G$ and that acts properly discontinuously on the affine space corresponding to~$V$.
\end{mainthm}
(Note that the choice of the representative~$\tilde{w}_0$ in the last condition does not matter, precisely because by assumption the vector~$v$ is fixed by~$L = Z_G(A)$.)

Let us give some comments about the conditions that we require on~$\rho$. Condition~\ref{itm:bad} here is the main one: this is just the negation of the condition that appears in the main theorem of~\cite{Smi16b}. The other two conditions are simplifying assumptions, that the author hopes to remove in the future. More specifically:
\begin{itemize}
\item Condition \ref{itm:split} is always satisfied if $G$ is split (indeed we then have $\mathfrak{l} = \mathfrak{a}$, and we may show that a vector is fixed by~$L$ (resp. $A$) if and only if it is fixed by its Lie algebra), but possibly also covers a few other cases.
\item Condition \ref{itm:non_swinging} is precisely the ``non-swinging'' assumption as introduced in~\cite{Smi16}. If $G$ is simple, then the only cases where this condition may fail are when its restricted root system is of type $A_n (n \geq 2)$, $D_{2n+1}$ or~$E_6$. (See also section~5.3 in~\cite{Smi16b} for a slightly more detailed discussion.)
\end{itemize}

\subsection{Strategy of the proof}
We proceed by contradiction: we suppose that $G \ltimes V$ contains a subgroup~$\Gamma$ with Zariski-dense linear part that acts properly discontinuously on~$V$.

We start, in Section~\ref{sec:non_radical}, by eliminating a trivial case: the case where the whole space of vectors fixed by~$A$ (also known as the zero restricted weight space $V^0$) is equal to zero. Starting from there, we always assume that this space is nontrivial.

We then heavily rely on the framework introduced in the author's previous paper~\cite{Smi16b}, that we briefly recall and slightly expand in Section~\ref{sec:generalized_schottky}. More specifically, we show (in Section~\ref{sec:subgroup_construction}) that such a group~$\Gamma$ necessarily contains a ``generalized Schottky'' subgroup, which satisfies verbatim all of the results proved in that paper, except for the last section. The key construction consists in associating, to every sufficiently nice element of $G \ltimes V$, some vector that is related to its translation part, called its \emph{Margulis invariant}. The most important result from~\cite{Smi16b} is then Proposition~10.2, which says that in this subgroup, the Margulis invariant of a generic word is roughly equal to the sum of the Margulis invariants of the letters.

However, at this point we have to diverge from the paper~\cite{Smi16b}. Indeed in Section~11 of~\cite{Smi16b}, we prove that the Margulis invariants of the elements of the group tend to infinity, and hence the group \emph{does} act properly. Here, on the contrary, we prove (in Section~\ref{sec:bounded_margulis}) that the Margulis invariants of the elements of the group accumulate in a bounded neighborhood of~$0$. For this, we use a method that is similar to the proof of Theorem~A in~\cite{AMS02}, but generalized to higher dimensions. In Section~\ref{sec:conclusion}, we then deduce from this fact that the group \emph{does not} act properly.

It is in this last deduction that we rely on conditions \ref{itm:split} and~\ref{itm:non_swinging} imposed on~$\rho$. The point of these two conditions is to ensure that the so-called \emph{quasi-translations} (as defined in Section~7.2 in~\cite{Smi16b}) are just ordinary translations (see Remark~\ref{quasi_translations_are_translations}). This in turn ensures that the Margulis invariant of an element contains all of the relevant information about its translation part (so that if the Margulis invariant is bounded, and the element has ``boundedly non-degenerate'' geometry, the translation part cannot escape to infinity).

\subsection{Acknowledgments}
I am very grateful to my Ph.D. advisor, Yves Benoist, who introduced me to this exciting and fruitful subject, and gave invaluable help and guidance in my initial work on this project.

I would also like to thank François Guéritaud for an interesting discussion a few years ago (in Luminy), that helped me shape this article in my mind long before I started actively working on it.

This work has been partially funded by the NSF grant DMS-1709952.

\section{Reduction to the case where $V^0 \neq 0$.}
\label{sec:non_radical}

We start by doing away with a trivial case.

\begin{proposition}
Suppose that $0$~is not a restricted weight of~$\rho$, \ie that $V^0 = 0$. Then there does not exist a subgroup~$\Gamma$ in the affine group~$G \ltimes V$ whose linear part is Zariski-dense in~$G$ and that acts properly discontinuously on the affine space corresponding to~$V$.
\end{proposition}
(See also Remark~3.5 in \cite{Smi16}.)

Note that this proposition actually gives a slight improvement of the Main Theorem: it tells us that at least if $V^0 = 0$ (which is a particular case of the condition~\ref{itm:bad}), we can dispense with the technical conditions \ref{itm:split} and~\ref{itm:non_swinging}.
\begin{proof}
By contradiction, let $\Gamma$ be such a group. By dimension arguments, it has a finitely-generated subgroup that is still Zariski-dense; so without loss of generality, we may assume that $\Gamma$~is finitely generated. We then have Selberg's lemma, which says that such a group $\Gamma$ is then virtually torsion-free, \ie contains a finite-index subgroup that is torsion-free. A finite-index subgroup of a Zariski-dense subgroup is still Zariski-dense; so without loss of generality, we may assume that $\Gamma$ is torsion-free.

Let $\tilde{G}$ be the set of elements of~$G$ that do not have~$1$ as an eigenvalue (when acting on~$V$). Clearly $\tilde{G}$ is Zariski-open in~$G$; and since $V^0 = 0$, this set is nonempty. Since by assumption $\ell(\Gamma)$ is Zariski-dense in~$G$, there exists an element $\gamma \in \Gamma$ whose linear part is in~$\tilde{G}$. This implies that $\gamma$ has a fixed point when acting on~$V_{\Aff}$ (but is not the identity). Since $\Gamma$~is torsion-free, $\gamma$ has infinite order. This contradicts properness of action.
\end{proof}

So from now on, we assume that this issue does not arise:
\begin{assumption}
\label{zero_is_a_weight}
From now on, we assume that $0$~is a restricted weight of~$\rho$:
\[\dim V^0 > 0.\]
\end{assumption}
(This is the same as Assumption~5.5 in~\cite{Smi16b}.)

In this case, we can say that all that is written in~\cite{Smi16b}, \emph{except} for the last section, still applies to our group $G \ltimes V$. From now on, we borrow all of the definitions and notations from \cite{Smi16b}.

\section{Generalized Schottky groups}
\label{sec:generalized_schottky}

The goal of this section is to recall the key lemma from~\cite{Smi16b}, namely Proposition~10.2, about ``generalized Schottky'' groups; and to formulate some easy corollaries of its proof.

\begin{definition}
For every generically symmetric and extreme vector~$X_0 \in \mathfrak{a}^+$ (see Section~5 in~\cite{Smi16b}) and for every constant $C \geq 1$, we say that a $k$-tuple $(g_1, \ldots, g_k)$ of elements of~$G \ltimes V$ is \emph{$C$-Schottky of type~$X_0$} if it satisfies the conditions of Proposition~10.2 in~\cite{Smi16b}. (Note that in that proposition, the definitions of $\rho$-regularity, $C$-non-degeneracy and of the Margulis invariant $M(g)$ all implicitly depend on the choice of~$X_0$.)

If $\Gamma$~is the group generated by this tuple, by abuse of terminology, we will also sometimes say that $\Gamma$ is \emph{$C$-Schottky of type~$X_0$}. Of course the reader has to keep in mind that this is not really a property of $\Gamma$ as an abstract group, but of $\Gamma$ together with a certain choice of a generating set.
\end{definition}

We may now restate Proposition~10.2 from~\cite{Smi16b}. Recall that:
\begin{itemize}
\item the notion of $\rho$-regularity is introduced in Definition~6.12 in~\cite{Smi16b};
\item $A^{\subgap}_g$ (resp. $A^{\sublap}_g$) is the \emph{affine ideally noncontracting} (resp.~\emph{nonexpanding}) \emph{space} corresponding to~$g$: see Definition~7.5 in~\cite{Smi16b};
\item \emph{$C$-non-degeneracy} is a quantitative measure of transversality: see Definition~7.20 in~\cite{Smi16b};
\item $s_{X_0}(g)$ is the \emph{contraction strength} of~$g$: see Definition~7.22 in~\cite{Smi16b};
\item $M(g)$ is the \emph{Margulis invariant} of~$g$, which measures its translation vector along some subspace fixed by the linear part of~$g$: see Definition~7.19 in~\cite{Smi16b}.
\end{itemize}
\begin{proposition}[\cite{Smi16b}]
\label{Schottky_properties}
Let $\Gamma = \langle g_1, \ldots, g_k \rangle$ be a $C$-Schottky subgroup of type $X_0$ in~$G \ltimes V$; and let $g = g_{i_1}^{\sigma_1} \cdots g_{i_l}^{\sigma_l}$ be any nonempty cyclically reduced word in~$\Gamma$. Then we have the following properties:
\begin{hypothenum}
\addtolength{\itemsep}{.5\baselineskip}
\item The map~$g$ is $\rho$-regular.
\item $\begin{cases}
       \alpha^\mathrm{Haus} \left(A^{\subgap}_{g},\;
                                  A^{\subgap}_{g_{i_1}^{\sigma_1}} \right)
          \lesssim_C 2 \left( 1 - 2^{-(l-1)} \right) s_{10.2}(C) \vspace{1mm} \\
       \alpha^\mathrm{Haus} \left(A^{\sublap}_{g},\;
                                  A^{\sublap}_{g_{i_l}^{\sigma_l}} \right)
          \lesssim_C 2 \left( 1 - 2^{-(l-1)} \right) s_{10.2}(C).
       \end{cases}$
\item $s_{X_0}(g) \leq 2^{-(l-1)} s_{10.2}(C)$.
\item $\displaystyle \left\| M(g) - \sum_{m=1}^{l} M(g_{i_m}^{\sigma_m}) \right\| \leq (l-1)\eps_{9.3}(2C)$.
\item If $h = g_{i'_1}^{\sigma'_1} \cdots g_{i'_{l'}}^{\sigma'_{l'}}$ is another nonempty cyclically reduced word such that $gh$ is also cyclically reduced, the pair $(g, h)$ is $2C$-non-degenerate.
\end{hypothenum}
\end{proposition}

Let us now present two corollaries of this proposition. In fact, both of them were actually used in the induction step when proving the proposition. (This is not circular reasoning: more precisely, the proof of point~(iv) for a word~$g = g'g''$ relied on Corollary \ref{Margulis_additivity} applied to shorter words $g'$ and $g''$; while the proof of point~(v) relied on Corollary \ref{different_first_letters_transverse} which itself relies only on point~(ii).)

\begin{corollary}
\label{Margulis_additivity}
Let $\Gamma = \langle g_1, \ldots, g_k \rangle$ be a $C$-Schottky subgroup of type $X_0$ in~$G \ltimes V$, and $g$ and~$h$ be two nonempty cyclically reduced words in~$\Gamma$ such that $g h$ is still cyclically reduced (without canceling any letters). Then we have
\[\|M(g h) - M(g) - M(h)\| \leq \eps_{9.3}(2C).\]
\end{corollary}
\begin{proof}
This follows immediately by plugging points (i), (iii) and~(v) into Proposition~9.3 from~\cite{Smi16b}.
\end{proof}

The notion of ``affine parabolic spaces'' that appears in the following corollary is introduced in Definition~7.13 from~\cite{Smi16b}.
\begin{corollary}
\label{different_first_letters_transverse}
Let $\Gamma = \langle g_1, \ldots, g_k \rangle$ be a $C$-Schottky subgroup of type $X_0$ in~$G \ltimes V$, and $g$ and~$h$ be two nonempty cyclically reduced words in~$\Gamma$ such that $h g$ is reduced. Then the pair of affine parabolic spaces
\[\left( A^\subgap_g,\; A^\sublap_h \right)\]
is $2C$-non-degenerate (and in particular transverse).
\end{corollary}
\begin{proof}
This follows by combining:
\begin{itemize}
\item point (ii) of the proposition applied to $g$ and~$h$,
\item and the second defining property of a Schottky tuple, \ie assumption~(H2) from Proposition~10.2 in~\cite{Smi16b};
\end{itemize}
and plugging them into Lemma~8.3 from~\cite{Smi16b}.
\end{proof}

We may use this to show that a pair of affine parabolic spaces coming from two elements of a Schottky group is always transverse, unless of course the two elements are inverse to each other (possibly up to some powers):

\begin{lemma}
\label{almost_always_transverse}
Let $\Gamma = \langle g_1, \ldots, g_k \rangle$ be a $C$-Schottky subgroup of type $X_0$ in~$G \ltimes V$, and $g$ and~$h$ be two nonempty cyclically reduced words in~$\Gamma$. Then the pair of affine parabolic spaces
\[\left( A^\subgap_g,\; A^\sublap_h \right)\]
is always transverse, unless $g$ and $h$ are of the form
\begin{equation}
\label{eq:g_h_proportional}
\begin{cases}
g = f^a \\
h = f^{-b}
\end{cases}
\end{equation}
for some $f \in \Gamma$ and some positive integers $a$ and~$b$.
\end{lemma}

\begin{proof}
Since $g$ and~$h$ are cyclically reduced words, we have
\[\operatorname{len}(g^n) = n\operatorname{len}(g)\]
for every nonnegative~$n$ (and similarly for~$h$). In particular the two words
\[g^{\operatorname{len}(h)} \quad\text{and}\quad (h^{-1})^{\operatorname{len}(g)}\]
have the same length, namely $\operatorname{len}(g)\operatorname{len}(h)$. (Of course instead of the product of the lengths, we could have used their least common multiple, or for that matter \emph{any} common multiple, or even infinite words formed by repeating $g$ and~$h^{-1}$ infinitely many times; the proof would still work the same way.) There are then two possible cases:
\begin{itemize}
\item Either the two words coincide: $g^{\operatorname{len}(h)} = (h^{-1})^{\operatorname{len}(g)}$. It is easy to see (basically by Euclid's GCD algorithm) that this is precisely equivalent to~\eqref{eq:g_h_proportional}.
\item Otherwise, the two words differ in at least one letter. Let then $p$~be their longest common prefix, and let~$g_i^\sigma$ (resp. $g_{i'}^{\sigma'}$) be the first letter of the word $g^{\operatorname{len}(h)}$ (resp. $(h^{-1})^{\operatorname{len}(g)}$) that follows~$p$. By assumption, we then have $(i, \sigma) \neq (i', \sigma')$.

Now consider the conjugates $p^{-1} g p$ and $p^{-1} h p$. If we write them in reduced form, we get cyclic permutations of~$g$ and~$h$ respectively (which are always cyclically reduced). Moreover, the reduced form of~$p^{-1} g p$ then starts with the letter~$g_i^\sigma$ and the reduced form of~$p^{-1} h p$ then ends with the letter~$g_{i'}^{-\sigma'}$. We may thus apply Corollary~\ref{different_first_letters_transverse} to the pair $(p^{-1} g p,\; p^{-1} h p)$, to conclude that the pair of affine parabolic spaces
\[\left( A^\subgap_{p^{-1} g p},\; A^\sublap_{p^{-1} h p} \right)\]
is transverse. Now given that the property of being transverse is invariant by the action of~$G$, we conclude that the pair
\[\left( A^\subgap_g,\; A^\sublap_h \right),\]
which is the image of the previous pair by~$p$, is transverse as well. \qedhere
\end{itemize}
\end{proof}

\section{Construction of a generalized Schottky subgroup in~$\Gamma$}
\label{sec:subgroup_construction}

We now introduce some subgroup $\Gamma \subset G \ltimes V$ whose linear part is Zariski-dense in~$G$. The remainder of this paper is dedicated to proving that its action on~$V_{\Aff}$ cannot be proper (which will prove the Main Theorem).

The goal of this section is to find, inside $\Gamma$, a $C$-Schottky subgroup of type~$X_0$, for a suitable generically symmetric and extreme vector~$X_0 \in \mathfrak{a}^+$ and constant $C \geq 1$. The actual construction will be done in subsection~\ref{sec:proper_construction}; before this, we need some preliminary work, to be done in subsection~\ref{sec:transverse_characterization}.

The results of this section hold for any representation~$\rho$ that has~$0$ as a restricted weight; we do not yet need the conditions \ref{itm:bad}, \ref{itm:split} and~\ref{itm:non_swinging} from the Main Theorem.

\subsection{Characterization of transversality of flags}
\label{sec:transverse_characterization}

The goal of this subsection is to prove Lemma~\ref{representations_transversality_condition}, which characterizes transverse pairs of flags. It encompasses Lemma~4.21~(ii) from~\cite{Smi16b}, but also provides its converse (which is what we will really need in the next subsection).

We fix, for the duration of this subsection, a vector $X \in \mathfrak{a}^+$. The result that we are going to prove holds without any additional assumption on~$X$. Also, it makes sense in a purely linear setting (\ie in the group~$G$ rather than $G \ltimes V$); so we may temporarily forget about our representation~$\rho$.

We start with the following lemma, which plays the role of Lemma~6.5 from~\cite{Smi16b} in the case where $\rho$ is replaced by the adjoint representation (so that $\Omega = \Sigma \cup \{0\}$), but $X$~is no longer assumed to be generic with respect to it (in other terms, we do not necessarily have $\Pi_X = \emptyset$). Define $\Sigma^\subg_X$, $\Sigma^\sube_X$ and~$\Sigma^\subl_X$ to be the set of restricted roots that take respectively positive, zero or negative values on~$X$.
\begin{lemma}
\label{Weyl_weak_stabilizers}
Every element $w$ of the restricted Weyl group~$W$ such that
\[w \Sigma^\subg_X \subset \Sigma^\subge_X\]
is actually an element of~$W_X$, the stabilizer of~$X$ (and in particular stabilizes both $\Sigma^\subg_X$ and~$\Sigma^\subge_X$).
\end{lemma}
\begin{proof}
Let $w \in W$ be an element satisfying this condition. The condition is equivalent to
\[\Sigma^\subg_X \subset w^{-1} \Sigma^\subge_X = \Sigma^\subge_{w X},\]
or in other terms
\begin{equation}
\label{eq:wX_is_X-positive}
\forall \alpha \in \Sigma^\subg_X,\quad \alpha(w X) \geq 0.
\end{equation}

Now note that $\Sigma^\sube_X$ is itself a (possibly empty) root system, whose Weyl group is precisely~$W_X$ (by Chevalley's lemma, see \eg \cite{Kna96}, Proposition~2.27) and which has $\Pi_X := \Pi \cap \Sigma^\sube_X$ as a system of positive simple roots. Hence there exists some $w' \in W_X$ such that $w' w X$ is dominant with respect to~$\Pi_X$, \ie on which every $\alpha \in \Pi_X = \Pi \cap \Sigma^\sube_X$ takes a nonnegative value.

On the other hand, since clearly $W_X$ stabilizes~$\Sigma^\subg_X$, this vector~$w' w X$ still satisfies \eqref{eq:wX_is_X-positive}; in particular every $\alpha \in \Pi \setminus \Pi_X = \Pi \cap \Sigma^\subg_X$ also takes a nonnegative value on~$w' w X$. We conclude that $w' w X$~is dominant with respect to all of~$\Pi$, \ie that $w' w X \in \mathfrak{a}^+$.

But every Weyl orbit intersects $\mathfrak{a}^+$ at exactly one point, so we actually have $w' w X = X$. Since by construction $w' \in W_X$, we conclude that $w \in W_X$ as well.
\end{proof}

We may now characterize $G$-orbits of pairs of flags in terms of the Bruhat decomposition of a suitable map. Recall Definition~2.16 from \cite{Smi16b} for the definition of the parabolic subgroups~$P^\pm_X$ and their algebras.

\begin{lemma}
\label{Bruhat_transversality_criterion}
Let $\phi_1, \phi_2 \in G$; consider the pair of cosets $(\phi_1 P^+_X, \phi_2 P^-_X) \in G/P^+_X \times G/P^-_X$. Then:
\begin{hypothenum}
\item its $G$-orbit depends only on the ``Bruhat projection''~$w$ of the map $\phi_1^{-1} \phi_2 w_0$, which is defined to be the unique element $w \in W$ such that
\[\phi_1^{-1} \phi_2 w_0 \in P^+ w P^+.\]
\item this pair lies in the same $G$-orbit as the pair $(P^+_X, P^-_X)$ if and only if this Bruhat projection~$w$ satisfies
\[w w_0 \in W_X.\]
\end{hypothenum}
\end{lemma}

\begin{proof}
For shortness' sake, we introduce the notation $Y := -w_0(X)$. (In subsequent applications, the vector~$X$ will actually be symmetric, meaning that $Y = X$).
\begin{hypothenum}
\item Since $P^+_X$ is by definition the normalizer of~$\mathfrak{p}^+_X$ (in the adjoint representation), the coset $\phi P^+_X$ is entirely determined by the conjugacy class $\Ad_{\phi} \mathfrak{p}_X^+$ and vice-versa:
\[\forall \phi, \phi' \in G,\qquad \phi P^+_X = \phi' P^+_X \iff \Ad_{\phi} \mathfrak{p}_X^+ = \Ad_{\phi'} \mathfrak{p}_X^+.\]
Of course the same statement holds for $P^-_X$ and $\mathfrak{p}^-_X$. So the question can be restated as follows: we need to prove that the $\Ad_G$-orbit
\[\Ad_G \cdot \left( \Ad_{\phi_1} \mathfrak{p}_X^+,\; \Ad_{\phi_2} \mathfrak{p}_X^- \right)\]
depends only on~$w$. Now noting that $\mathfrak{p}_X^- = \Ad_{w_0} \mathfrak{p}_Y^+$, and multiplying everything by $\phi_1^{-1}$, this orbit is in fact equal to the orbit
\[\Ad_G \cdot \left( \mathfrak{p}_X^+,\; \Ad_{\phi^{-1} \phi_2 w_0} \mathfrak{p}_Y^+ \right).\]
Now by definition of~$w$, we can find some $p_1, p_2 \in P^+$ such that $\phi^{-1} \phi_2 w_0 = p_1 w p_2$. Let us multiply everything by~$p_1^{-1}$. The elements $p_1^{-1}$ and~$p_2$ are by assumption in~$P^+$, which is a subset of both $P_X^+$ and~$P_Y^+$ and in particular normalizes both of their Lie algebras. So both rightmost factors can be absorbed into $\mathfrak{p}_X^+$ and $\mathfrak{p}_Y^+$ respectively, and we see that our orbit is the same as the orbit
\[\Ad_G \cdot \left( \mathfrak{p}_X^+,\; \Ad_{w} \mathfrak{p}_Y^+ \right),\]
and indeed depends only on~$w$.
\item By (i), it is enough to show that, for every $w \in W$, the pair $(P_X^+, w w_0 P_X^-)$ is in the same $G$-orbit as the pair $(P_X^+, P_X^-)$ if and only if $w w_0 \in W_X$. Applying once again the reductions made in the proof of~(i), this can be further reduced to showing that, for every $w \in W$, we have
\[\left( \mathfrak{p}_X^+,\; \Ad_w \mathfrak{p}_Y^+ \right) \in \Ad_G \cdot \left( \mathfrak{p}_X^+,\; \mathfrak{p}_X^- \right)
\quad \iff \quad w w_0 \in W_X.\]
Indeed:
\begin{itemize}
\item Assume first that $w w_0 \in W_X$. This means in particular that $w w_0$ stabilizes $\mathfrak{p}_X^-$, so that we actually have
\[\left( \mathfrak{p}_X^+,\; \Ad_w \mathfrak{p}_Y^+ \right) = \left( \mathfrak{p}_X^+,\; \mathfrak{p}_X^- \right);\]
of course the orbits are then equal as well.
\item Conversely, assume that $\left( \mathfrak{p}_X^+,\; \Ad_w \mathfrak{p}_Y^+ \right) \in \Ad_G \cdot \left( \mathfrak{p}_X^+,\; \mathfrak{p}_X^- \right)$. This means in particular that the pair $\left( \mathfrak{p}_X^+,\; \Ad_w \mathfrak{p}_Y^+ \right)$ is transverse, in the sense that
\[\mathfrak{p}_X^+ + \Ad_w \mathfrak{p}_Y^+ = \mathfrak{g}.\]
This is equivalent to
\[\Sigma_X^\subge \cup w \Sigma_Y^\subge = \Sigma,\]
which in turn is equivalent to
\[w w_0 \Sigma_X^\subg \subset \Sigma_X^\subge\]
(compare this with~(7.7) in~\cite{Smi16b}). We conclude by Lemma~\ref{Weyl_weak_stabilizers}.
\qedhere
\end{itemize}
\end{hypothenum}
\end{proof}

Recall that the proof of Lemma~4.21 in~\cite{Smi16b} relied on a ``model'' attracting line and repelling hyperplane in each of the ``reference'' representations; let us now give a notation to these pairs of spaces.

\begin{definition}
Let $i \in \Pi$, and let $(\rho_i, V_i)$ be one of the representations of~$G$ introduced in Proposition~2.12 from~\cite{Smi16b}. We call $V^s_{i, 0}$ the highest restricted weight space of~$V_i$, and $V^u_{i, 0}$ its natural complement:
\[V^s_{i, 0} := V_i^{n_i \varpi_i};\]
\[V^u_{i, 0} := \bigoplus_{\lambda \neq n_i \varpi_i} V_i^\lambda.\]
\end{definition}

We are now ready to state, and prove, the announced lemma.

\begin{lemma}
\label{representations_transversality_condition}
For any $\phi_1, \phi_2 \in G$, the pair $(\phi_1 P^+_X, \phi_2 P^-_X) \in G/P^+_X \times G/P^-_X$ is transverse if and only if, for every $i \in \Pi \setminus \Pi_X$, we have $\phi_1 V^s_{i, 0} \not\in \phi_2 V^u_{i, 0}$.
\end{lemma}
Of course it is understood here that $\phi_1$ and~$\phi_2$ act on~$V_i$ by~$\rho_i$.
\begin{proof}
Let $\phi_1, \phi_2 \in G$, and let $w$~be the Bruhat projection of the map $\phi_1^{-1} \phi_2 w_0$. By Lemma~\ref{Bruhat_transversality_criterion}, the pair~$(\phi_1 P^+_X, \phi_2 P^-_X)$ is transverse if and only if $w w_0 \in W_X$.

On the other hand, for every $i \in \Pi \setminus \Pi_X$, the condition
\[\phi_1 V^s_{i, 0} \not\in \phi_2 V^u_{i, 0}\]
is equivalent to
\[n_i \varpi_i \not\in w \Big( \Omega_i \setminus \{w_0(n_i \varpi_i)\} \Big)\]
(where $\Omega_i$, as in~\cite{Smi16b}, denotes the set of restricted weights of~$\rho_i$). This happens if and only if $w w_0$ fixes~$n_i \varpi_i$, \ie if and only if $w w_0 \in W_{\varpi_i}$.

Finally, we note that
\[W_X = \bigcap_{i \in \Pi \setminus \Pi_X} W_{\varpi_i}\]
(see (4.13) in~\cite{Smi16b}). The conclusion follows.
\end{proof}

\subsection{Construction of the subgroup}
\label{sec:proper_construction}

We are now ready to construct a generalized Schottky subgroup in~$\Gamma$. Our first step is to find an appropriate vector~$X_0$ (that we will fix for the remainder of the paper).

\begin{proposition}
There exists some generically symmetric and extreme vector $X_0 \in \mathfrak{a}^+$ such that $\Gamma$ contains at least one element~$g$ compatible with~$X_0$. 
\end{proposition}
\begin{proof}
Define the \emph{limit cone} of $\Gamma$ (denoted by $\ell_\Gamma$; not to be confused with $\ell(\Gamma)$, which is the linear part of~$\Gamma$) to be the smallest closed cone in $\mathfrak{a}^+$ containing the Jordan projections of all the elements of~$\Gamma$. (The Jordan projection, also known as the Lyapunov projection, is the map $\jordan: G \to \mathfrak{a}^+$ given by Definition~2.3 in~\cite{Smi16b}, or equivalently the map $\log \lambda$ in the notations of~\cite{Ben97}.) Then Theorem~1.2.a.$\beta$ in~\cite{Ben97} says that $\ell_\Gamma$ is convex and has nonempty interior. In particular its intersection with the $(-w_0)$-invariant subspace of $\mathfrak{a}^+$ also has nonempty interior: hence it contains at least one generically symmetric vector $X'_0$. Applying Proposition~5.13 from~\cite{Smi16b}, we can then find an extreme vector~$X_0$ of the same type as~$X'_0$.

Now consider the set $\mathfrak{a}'_{\rho, X_0}$ of vectors in~$\mathfrak{a}$ compatible with~$X_0$, introduced in Remark~6.15 in~\cite{Smi16b}. This set is an open convex cone containing~$X'_0$. Now we have $X'_0 \in \ell_\Gamma$, which means by definition that $\Gamma$ contains elements whose Jordan projections have direction arbitrarily close to~$X'_0$. In particular $\Gamma$ must contain some element~$g$ whose Jordan projection is in~$\mathfrak{a}'_{\rho, X_0}$, \ie which is compatible with~$X_0$.
\end{proof}

We fix this value of~$X_0$ for the remainder of the paper. Let us now find a \emph{second} element~$h \in \Gamma$, also compatible with~$X_0$, and ``in general position'' with respect to~$g$:

\begin{proposition}
\label{construction_g_h}
Let $g \in \Gamma$~be an element compatible with~$X_0$. Then there exists another element $h \in \Gamma$ compatible with~$X_0$ such that both attracting flags $y^{X_0, +}_h$ and $y^{X_0, +}_{h^{-1}}$ are transverse to both repelling flags $y^{X_0, -}_g$ and $y^{X_0, -}_{g^{-1}}$.
\end{proposition}
(See Definition~4.3 in~\cite{Smi16b} for the definitions of the attracting and repelling flags. They make sense when $g$ and~$h$ are $X_0$-regular; but we know that $g$ and~$h$ are compatible with~$X_0$, which, by Proposition~6.16~(i) in~\cite{Smi16b}, is a stronger property.)
\begin{proof}
We find $h$ as the conjugate of~$g$ by some element $\phi \in \Gamma$; this automatically ensures that $h$~is still compatible with~$X_0$. The transversality conditions are then satisfied if and only if the element~$\phi$ is contained in each of the following four sets:
\[U^{++} := \setsuch{\phi \in G}{\phi \left( y^{X_0, +}_{g} \right) \text{ transverse to } y^{X_0, -}_{g}};\]
\[U^{+-} := \setsuch{\phi \in G}{\phi \left( y^{X_0, +}_{g} \right) \text{ transverse to } y^{X_0, -}_{g^{-1}}};\]
\[U^{-+} := \setsuch{\phi \in G}{\phi \left( y^{X_0, +}_{g^{-1}} \right) \text{ transverse to } y^{X_0, -}_{g}};\]
\[U^{--} := \setsuch{\phi \in G}{\phi \left( y^{X_0, +}_{g^{-1}} \right) \text{ transverse to } y^{X_0, -}_{g^{-1}}}.\]
It remains to show that $\Gamma \cap U^{++} \cap U^{+-} \cap U^{-+} \cap U^{--} \neq \emptyset$. Since $\Gamma$~is by assumption Zariski-dense, the result will follow if we can prove that all four sets $U^{++}$, $U^{+-}$, $U^{-+}$, $U^{--}$ are Zariski-open and nonempty.

From Lemma~\ref{representations_transversality_condition}, it follows (using Lemma~4.21 from~\cite{Smi16b}) that we have
\[U^{++} = \bigcap_{i \in \Pi \setminus \Pi_X} \rho_i^{-1} \setsuch{f \in \GL(V_i)}{f E^s_{\rho_i(g)} \not\in E^u_{\rho_i(g)}}.\]
and similarly for the other three sets. Now each set $\setsuch{f \in \GL(V_i)}{f E^s_{\rho_i(g)} \not\in E^u_{\rho_i(g^{-1})}}$ is just the (set-theoretical) complement of a hyperplane, hence Zariski-open. Since the representations $\rho_i$ are algebraic, it follows that $U^{++}$ (and similary $U^{+-}$, $U^{-+}$ and $U^{--}$) is Zariski-open.

Now the sets $U^{++}$ and $U^{--}$ are obviously nonempty, as they contain the identity. For $U^{+-}$ and~$U^{-+}$, it is easy to show nonemptiness by using the transversality criterion of Lemma~\ref{Bruhat_transversality_criterion}.
\end{proof}

This construction ensures the following property:
\begin{lemma}
\label{pairwise_C_non_deg}
Let $g$ and~$h$ be two elements compatible with~$X_0$ and satisfying the transversality conditions of Proposition~\ref{construction_g_h}. Then there exists a finite constant $C \geq 1$ such that for any two (not necessarily distinct) elements $f_1, f_2 \in \{g, g^{-1}, h, h^{-1}\}$ such that $f_1 \neq f_2^{-1}$, the pair of affine parabolic spaces
\[\left( A^\subgap_{f_1},\; A^\sublap_{f_2} \right)\]
is $C$-non-degenerate.
\end{lemma}
\begin{proof}
Let $f_1$ and~$f_2$ be two such elements. Then the pair of flags
\[\left( y^{X_0, +}_{f_1},\; y^{X_0, -}_{f_2} \right)\]
is transverse: indeed for $f_1 \neq f_2$ this is true by assumption, and for $f_1 = f_2$ this follows automatically simply because $f_{1 \text{ or } 2}$ is $X_0$-regular. By Remark~7.15 from~\cite{Smi16b}, this is equivalent to saying that for any such $f_1, f_2$, the pair
\[\left( V^\subgap_{f_1},\; V^\sublap_{f_2} \right)\]
is transverse. This, in turn, is obviously equivalent to the pair
\[\left( A^\subgap_{f_1},\; A^\sublap_{f_2} \right)\]
being transverse. Now clearly, any transverse pair of affine parabolic spaces is $C$-non-degenerate for some finite $C \geq 1$; so it suffices to take the largest among these constants~$C$.
\end{proof}

We may now finally construct the desired subgroup.

\begin{definition}~
\begin{itemize}
\item For the rest of the paper, we fix two elements $g$ and $h$ compatible with~$X_0$ and satisfying the transversality conditions of Proposition~\ref{construction_g_h}.
\item We also fix a value $C \geq 1$ satisfying Lemma~\ref{pairwise_C_non_deg}.
\item Finally we call $\Gamma'$ the group generated by $g^N$ and~$h^N$, where $N$~is a positive integer large enough that the pair $(g^N, h^N)$ is $C$-Schottky of type~$X_0$. Concretely, we need to choose $N$ such that for every $f \in \{g, g^{-1}, h, h^{-1}\}$, we have
\[s_{X_0}\left( f^N \right) \leq s_{10.2}(C)\]
(where $s_{10.2}(C)$~is the constant introduced in Proposition~10.2 from~\cite{Smi16b}). This is made possible by Proposition~7.23.(ii) from~\cite{Smi16b}, which ensures that the left-hand side tends to~$0$ as $N$ tends to infinity.
\end{itemize}
\end{definition}

\section{Construction of a sequence with bounded Margulis invariants}
\label{sec:bounded_margulis}

The goal of this section is to prove Proposition~\ref{infinitely_many_bounded_invariants}, which gives, in this generalized Schottky group~$\Gamma'$, an infinite collection of elements whose Margulis invariants remain bounded. Actually, we will find these elements in an even smaller subgroup $\Gamma'' \subset \Gamma'$, that is still generalized Schottky but possibly with a different parameter.

\begin{assumption}
From now on, we assume that $\rho$ satisfies condition~\ref{itm:bad} from the Main Theorem. Using the notation introduced in Proposition~7.8 in~\cite{Smi16b}, it can now be rephrased as follows:
\begin{equation}
\forall v \in V_0^\transl,\quad -w_0(v) = -v.
\end{equation}
\end{assumption}

Note that conditions \ref{itm:split} and~\ref{itm:non_swinging} from the Main Theorem are not required until the next section.

With this assumption, Proposition~9.1 from~\cite{Smi16b} reduces to a particularly simple form:
\begin{remark}
\label{opposite_margulis_invariant}
For every $\rho$-regular map $g \in G \ltimes V$, we then have
\begin{equation}
\label{eq:opposite_margulis_invariant}
M(g^{-1}) = -M(g).
\end{equation}
In other terms, the identity $M(g^n) = n M(g)$ now holds for all integer values of~$n$, positive \emph{and negative}.
\end{remark}

\begin{proposition}
\label{infinitely_many_bounded_invariants}
There exists a constant $C' > 1$, an integer $k \geq 2$ and elements $\gamma_1, \ldots, \gamma_k \in \Gamma'$ with the following properties:
\begin{itemize}
\item The family $(\gamma_1, \ldots, \gamma_k)$ is $C'$-Schottky of type $X_0$.
\item The group $\Gamma''$ generated by this family contains an infinite subset $S$ of elements that are cyclically reduced (as words in~$\Gamma''$) and such that
\[\exists R \geq 0,\; \forall \gamma \in S,\quad \|M(\gamma)\| \leq R.\]
\end{itemize}
\end{proposition}

To prove this, we distinguish two cases: either the Margulis invariants of the elements of~$\Gamma'$ are all collinear, or they span a vector subspace of~$V_0^\transl$ of dimension at least~$2$.

\subsection{Case where $M(\Gamma')$ is contained in a line}
\label{sec:collinear}

In this case, we can basically apply the same techniques as for the proof of Theorem~A in~\cite{AMS02}.

\begin{proof}[Proof of Proposition~\ref{infinitely_many_bounded_invariants} when $\dim \Span(M(\Gamma')) \leq 1$.]
In this case, restricting to a smaller group is unnecessary: we simply take $C' := C$, $k := 2$, $\gamma_1 := g$ and~$\gamma_2 := h$ (so that $\Gamma'' = \Gamma'$).

By assumption, the vectors $M(g)$ and~$M(h)$ must in particular be linearly dependent. Without loss of generality (exchanging $g$ and~$h$ if needed), suppose that we have
\[M(g) = c M(h)\]
for some $c \in \mathbb{R}$. Now we deduce from Proposition~9.3 in~\cite{Smi16b} that for every natural integer~$n$, we have
\begin{align*}
\|M(g^n h^{- \lfloor c n \rfloor})\| &\leq \|n M(g) - \lfloor c n \rfloor M(h)\| + \eps_{9.3}(C) \\
                                     &\leq \|M(h)\| + \eps_{9.3}(C).
\end{align*}
(To ensure that this works no matter the sign of~$c$, we rely on~\eqref{eq:opposite_margulis_invariant}.) Thus taking $R$ to be the right-hand side of the last inequality, the set
\[S = \setsuch{g^n h^{- \lfloor c n \rfloor}}{n \in \mathbb{N}}\]
satisfies the required conditions.
\end{proof}

\subsection{Case where $M(\Gamma')$ is not contained in a line: construction of the subgroup}
\label{sec:non_collinear_subgroup}

Let us now assume that $\dim \Span(M(\Gamma')) \geq 2$. We split the proof of this case into two parts: in this subsection we construct the generalized Schottky subgroup $\Gamma'' \subset \Gamma'$, and in the next subsection we construct the set~$S$ inside it.

The number of generators~$k$ of the subgroup will be taken to be
\[k := \dim \Span(M(\Gamma')).\]
Let us choose once and for all
\[g_1, \ldots, g_k\]
some cyclically reduced elements of~$\Gamma'$ whose Margulis invariants form a basis of the vector subspace spanned by~$M(\Gamma')$. (We may assume them to be cyclically reduced since, by construction, the Margulis invariant of an element only depends on its conjugacy class.)

We start by checking that these elements and their inverses are pairwise in general position:
\begin{lemma}
\label{new_generators_transverse}
For any two indices $i, i' \in \{1, \ldots, k\}$ and signs $\sigma$, $\sigma'$ such that $(i', \sigma') \neq (i, -\sigma)$, the pair of affine parabolic spaces
\[\left( A^\subgap_{g_i^\sigma},\; A^\sublap_{g_{i'}^{\sigma'}} \right)\]
is transverse.
\end{lemma}
This is an easy consequence of Lemma~\ref{almost_always_transverse}.
\begin{proof}
By contradiction, let $i, i', \sigma, \sigma'$ be some indices and signs such that the pair is \emph{not} transverse. From Lemma~\ref{almost_always_transverse}, this is only possible if we have
\[\begin{cases}
g_i^\sigma = f^a \\
g_{i'}^{\sigma'} = f^{-b}
\end{cases}\]
for some $f \in \Gamma'$ and some positive integers $a$ and~$b$. Using the identity~\eqref{eq:opposite_margulis_invariant}, this means that the Margulis invariants of the maps $g_i$ and~$g_{i'}$ are related by
\[M(g_i) = -\frac{\sigma' a}{\sigma b}M(g_{i'}).\]
By assumption the Margulis invariants of the different maps $g_i$ are linearly independent; so we must have $i = i'$. If we additionally had $\sigma = \sigma'$, then the two parabolic spaces would be the dynamic spaces of one single $\rho$-regular map, so they would be transverse. So we necessarily have $\sigma = -\sigma'$.
\end{proof}

This allows us to construct our subgroup:
\begin{definition}~
\begin{itemize}
\item In the light of Lemma~\ref{new_generators_transverse}, we fix a constant~$C' \geq 1$ such that all of the pairs of spaces concerned by this lemma (there are $(2k)^2 - 2k$ of them) are $C'$-non-degenerate.
\item We call $\phi_{\Gamma'}$ the linear map that maps any vector in~$\Span(M(\Gamma'))$ to its coordinates in the basis $(M(g_1), \ldots, M(g_k))$.
\item We fix an integer $N'$ large enough that, for every $i = 1, \ldots, k$ and $\sigma = \pm 1$, we have:
\[s_{X_0}(g_i^{\sigma N'}) \leq s_{10.2}(C'),\]
where $s_{10.2}$~is the constant from Proposition~10.2 in~\cite{Smi16b}. This is possible thanks to Proposition~7.23.(ii) from~\cite{Smi16b}. Additionally, we require $N'$ to satisfy
\begin{equation}
\label{eq:N_lower_bound}
N' \geq 12 \sqrt{k} \left\|\phi_{\Gamma'}\right\| \eps_{9.3}(2C'),
\end{equation}
where $\eps_{9.3}$~is the constant from Proposition~9.3 in~\cite{Smi16b}.
\item Finally, for all $i = 1, \ldots, k$, we set $\gamma_i := g_i^{N'}$; and we set $\Gamma'' := \langle \gamma_1, \ldots, \gamma_k \rangle$. Then the first assumption on $N'$ ensures that the group $\Gamma''$ is indeed $C'$-Schottky of type~$X_0$. (The second assumption basically ensures that the Margulis invariants of its generators are large enough that the error term in Proposition~\ref{Margulis_additivity}.(iv) becomes negligible.)
\end{itemize}
\end{definition}

\subsection{Case where $M(\Gamma')$ is not contained in a line: construction of the sequence}
\label{sec:non_collinear_sequence}

It now remains to construct an infinite subset $S \subset \Gamma''$ of elements whose Margulis invariants remain bounded.

The basic idea is as follows: start with an arbitrary prefix $w$; then we can always complete it to a word whose Margulis invariant is in some fixed compact set. Indeed thanks to~\eqref{eq:opposite_margulis_invariant}, no matter where we are in the vector subspace $\Span(M(\Gamma'))$, we can always find a letter among the generators and their inverses whose Margulis invariant points roughly towards the origin. Then ``approximate additivity'' of Margulis invariants ensures that, when we multiply by this letter, the norm of the Margulis invariant decreases (or stays small). This is roughly the content of Lemma~\ref{can_always_reduce_norm} below.

There is however a complication: we need to ensure, at every step, that the word remains cyclically reduced. So every time we append a letter, we might need to add some ``padding'' to protect it from possible cancellations. The following technical lemma tells us that even with this ``padding'', we can still manage to decrease the norm of the vector (provided that we choose the ``padding'' wisely, and unless the vector was already small to begin with).

\begin{lemma}
\label{vector_norm_inequality}
Take any vector $\alpha = (c_1, \ldots, c_k) \in \mathbb{R}^k$ with sufficiently large Euclidean norm, namely
\[\textstyle \|\alpha\| \geq 5\sqrt{k} + \frac{\sqrt{5}}{2}.\]
Let $i$~be the index of the component having the largest absolute value, and let $\sigma \in \{\pm 1\}$ be its sign, so that we have
\[\sigma c_i = |c_i| = \max_{j = 1, \ldots, k} |c_j|.\]
Let $j$ be any other index (here we use the assumption that $k \geq 2$), and let $\tau \in \{\pm 1\}$ be such that $\tau c_j$~is nonnegative. Finally, let $\beta = \sigma e_i + x \tau e_j$, where $x$ is either $0$, $1$ or~$2$ (and $(e_1, \ldots, e_k)$ stands for the canonical basis of~$\mathbb{R}^k$).

Then we have
\[\|\alpha - \beta\| \leq \|\alpha\| - \frac{1}{2\sqrt{k}}.\]
\end{lemma}

\begin{proof}
Note that from the definition of~$c_i$, it follows that $\|\alpha\| \leq \sqrt{k} |c_i|$. In particular we then have
\[\textstyle |c_i| \geq \frac{1}{\sqrt{k}}\|\alpha\| \geq 5 + \frac{\sqrt{5}}{2\sqrt{k}}.\]
Now we compute:
\begin{align*}
\|\alpha\| - \|\alpha - \beta\| &=    \frac{\|\alpha\|^2 - \|\alpha - \beta\|^2}{\|\alpha\| + \|\alpha - \beta\|} \\
                                &=    \frac{2(|c_i| + x |c_j|) - (1 + x^2)}{\|\alpha\| + \|\alpha - \beta\|} \\
                                &\geq \frac{2 |c_i| - 5}{2\|\alpha\| + \|\beta\|} & \text{ (this works because $2 |c_i| - 5 \geq 0$)} \\
                                &\geq \frac{2 |c_i| - 5}{2\sqrt{k}|c_i| + \sqrt{5}} \\
                                &=    \frac{1}{\sqrt{k}} - \frac{5 + \frac{\sqrt{5}}{\sqrt{k}}}{2\sqrt{k}|c_i| + \sqrt{5}} \\
                                &\geq \frac{1}{\sqrt{k}} - \frac{5 + \frac{\sqrt{5}}{\sqrt{k}}}{2\sqrt{k}\left(5 + \frac{\sqrt{5}}{2\sqrt{k}}\right) + \sqrt{5}} \\
                                &= \frac{1}{2\sqrt{k}}. & \qedhere
\end{align*}
\end{proof}

The following lemma now shows how we can decrement the Margulis invariant of a word by appending one, two or three letters.

\begin{lemma}
\label{can_always_reduce_norm}
Let $w$~be any cyclically reduced word on the generators $\gamma_1, \ldots, \gamma_k$, such that
\begin{equation}
\label{eq:N_alpha_lower_bound}
\|\phi_{\Gamma'}(M(w))\| \geq \left(5\sqrt{k} + \frac{\sqrt{5}}{2}\right)N'.
\end{equation}
Then there exists a cyclically reduced word~$u$ on the same generators such that $w u$ is still cyclically reduced (without canceling any letters), and we have
\[\|\phi_{\Gamma'}(M(w u))\| \leq \|\phi_{\Gamma'}(M(w))\| - \frac{N'}{4\sqrt{k}}.\]
\end{lemma}
\begin{proof}
Let $c_1, \ldots, c_k$ be the coordinates of~$M(w)$ in the basis~$(M(g_1), \ldots, M(g_k))$. Let $i$~be the index of the one that has the largest absolute value, and let $\sigma \in \{\pm 1\}$ be its sign, so that we have
\[\sigma c_i = |c_i| = \max_{j = 1, \ldots, k} |c_j|.\]
Let $j$ be any other index (here we use the assumption that $k \geq 2$), and let $\tau \in \{\pm 1\}$ be such that $\tau c_j$~is nonnegative.

We now set $u = l_1 \gamma_i^{-\sigma} l_2$, where:
\[l_1 := \begin{cases}
\gamma_j^{-\tau} &\text{if the last letter of $w$ is } \gamma_i^\sigma, \\
1                &\text{if it is anything else;}
\end{cases}\]
\[l_2 := \begin{cases}
\gamma_j^{-\tau} &\text{if the first letter of $w$ is } \gamma_i^\sigma, \\
1                &\text{if it is anything else.}
\end{cases}\]
In other terms, $u$ is equal to either $\gamma_i^{-\sigma}$, $\gamma_j^{-\tau} \gamma_i^{-\sigma}$, $\gamma_i^{-\sigma} \gamma_j^{-\tau}$ or $\gamma_j^{-\tau} \gamma_i^{-\sigma} \gamma_j^{-\tau}$, depending on the first and last letter of~$w$. Clearly $u$ is cyclically reduced in all four cases; and its construction ensures that $w u$ is still cyclically reduced.

By Corollary~\ref{Margulis_additivity} and Remark~\ref{opposite_margulis_invariant}, we then have
\[M(w u) = M(w) - \sigma M(\gamma_i) - x \tau M(\gamma_j) + E \quad\text{ with }\quad \|E\| \leq 3\eps_{9.3}(2C'),\]
where $x = 0, 1$ or~$2$ depending on the first and last letter of~$w$. Applying $\phi_{\Gamma'}$ to this estimate, we get
\[\phi_{\Gamma'}(M(w u)) = \phi_{\Gamma'}(M(w)) - N'(\sigma e_i + x \tau e_j) + E' \quad\text{ with }\quad \|E'\| \leq 3\left\|\phi_{\Gamma'}\right\|\eps_{9.3}(2C').\]
Now recall the lower bound~\eqref{eq:N_alpha_lower_bound} that we have on $\phi_{\Gamma'}(M(w))$. This allows us to apply Lemma~\ref{vector_norm_inequality}, rescaled by~$N'$, to the right-hand side of this formula (error term~$E'$ excluded). We obtain that
\begin{align*}
\|\phi_{\Gamma'}(M(w u))\| &\leq \|\phi_{\Gamma'}(M(w))\| - \frac{1}{2\sqrt{k}}N' + \|E'\| \\
                                  &\leq \|\phi_{\Gamma'}(M(w))\| - \frac{1}{2\sqrt{k}}N' + 3\left\|\phi_{\Gamma'}\right\|\eps_{9.3}(2C').
\end{align*}
Now recall that $N'$ has been chosen to satisfy the lower bound~\eqref{eq:N_lower_bound}; this translates to
\[3\left\|\phi_{\Gamma'}\right\|\eps_{9.3}(2C') \leq \frac{1}{4\sqrt{k}}N'.\]
The conclusion follows.
\end{proof}

\begin{proof}[Proof of Proposition~\ref{infinitely_many_bounded_invariants} when $\dim \Span(M(\Gamma')) \geq 2$.]
By applying Lemma~\ref{can_always_reduce_norm} iteratively, we see that for any cyclically reduced word~$w$ in~$\Gamma''$, there is another cyclically reduced word~$w'$ having~$w$ as a prefix and whose Margulis invariant is bounded by
\[\|\phi_{\Gamma'}(M(w'))\| \leq \left(5\sqrt{k} + \frac{\sqrt{5}}{2}\right)N'.\]
Let $S$ be the set of all such words~$w'$. By construction these words can have \emph{any} cyclically reduced word as a prefix, so clearly $S$ is infinite. Moreover, the Margulis invariants of its elements are all bounded by the constant
\[R = \|\phi_{\Gamma'}^{-1}\|\left(5\sqrt{k} + \frac{\sqrt{5}}{2}\right)N',\]
as required.
\end{proof}

\section{Proof of non-properness}
\label{sec:conclusion}

From boundedness of Margulis invariants, we may now deduce that the group does not act properly.

To do this, we have to rely on all the conditions of the Main Theorem:
\begin{assumption}
In addition to condition~\ref{itm:bad}, the representation~$\rho$ also satisfies conditions \ref{itm:split} and~\ref{itm:non_swinging} of the Main Theorem.
\end{assumption}
\begin{remark}
\label{quasi_translations_are_translations}
Using notations from Section~7.2 from~\cite{Smi16b}, these two conditions are easily seen to be equivalent to the statement
\begin{equation}
\label{eq:translation_is_all}
V^\transl_0 = V^0 = V^\sube_0,
\end{equation}
or equivalently
\begin{equation}
\label{eq:extra_is_nothing}
V^\extra_0 = V^\subenz \oplus (V^\extra_0 \cap V^0) = 0
\end{equation}
(where $V^\subenz$, defined in~(7.3) in~\cite{Smi16b}, is roughly the part of~$V^\sube_0$ other than $V^0$). In fact, \ref{itm:split} and~\ref{itm:non_swinging} respectively account for each for the two equalities in~\eqref{eq:translation_is_all}, or for the vanishing of each direct summand in~\eqref{eq:extra_is_nothing}.

Plugging this into Proposition~7.8 in~\cite{Smi16b}, the two conditions together mean that quasi-translations are actually simply translations.
\end{remark}

\begin{proof}[Proof of the Main Theorem.]
We will now deduce from Proposition~\ref{infinitely_many_bounded_invariants} that $\Gamma''$ (hence, \emph{a fortiori}, the larger group~$\Gamma$) does not act properly on~$V_{\Aff}$.

We introduce the following compact subset of~$V_{\Aff}$:
\begin{align*}
K &:= B_A\Big(0,\; 2C'\sqrt{R^2 + 1}\Big) \cap V_{\Aff} \\
  &= B_{V_{\Aff}}\Big(p_0,\; \sqrt{(2C')^2(R^2 + 1) - 1}\Big),
\end{align*}
where $p_0$~is the chosen origin of~$V_{\Aff}$ (see Section~6.2 in~\cite{Smi16b}). We claim that
\[\forall \gamma \in S,\quad \gamma K \cap K \neq \emptyset;\]
since $S \subset \Gamma'' \subset \Gamma' \subset \Gamma$~is infinite, this completes the proof.

Indeed, take any $\gamma \in S \subset \Gamma''$. By construction, $\gamma$~is then a cyclically reduced word in~$\Gamma''$; in particular, by Proposition~10.2 in~\cite{Smi16b}, it is then $\rho$-regular and $2C'$-non-degenerate. Let $\phi_\gamma$~be an optimal canonizing map of~$\gamma$ (see Definition~7.20 in~\cite{Smi16b}), so that we have
\[\|\phi_\gamma\| \leq 2C' \quad\text{and}\quad \|\phi_\gamma^{-1}\| \leq 2C'.\]
In particular, this implies that
\begin{align}
\label{eq:contains_R_ball}
\phi_\gamma(K) &= \phi_\gamma \left( B_A\Big(0,\; 2C'\sqrt{R^2 + 1}\Big) \right) \cap V_{\Aff} \nonumber \\
                    &\supset B_A\Big(0,\; \sqrt{R^2 + 1}\Big) \cap V_{\Aff} \nonumber \\
                    &= B_{V_{\Aff}}(p_0, R).
\end{align}
Now recall (Proposition~7.10 from~\cite{Smi16b}) that the conjugate map $\phi_\gamma \circ \gamma \circ \phi_\gamma^{-1}$ acts on the space $A^\sube_0$ (which by definition contains $p_0$) by quasi-translation. But given the assumptions we made on~$\rho$, a quasi-translation is just a translation (see Remark~\ref{quasi_translations_are_translations} above). Moreover, by definition, the translation vector of the map~$\phi_\gamma \circ \gamma \circ \phi_\gamma^{-1}$ is then precisely equal to~$M(\gamma)$. Thus we have
\[\phi_\gamma \circ \gamma \circ \phi_\gamma^{-1} (p_0) = p_0 + M(\gamma).\]
Now since by assumption, we have $\|M(\gamma)\| \leq R$, it follows from~\eqref{eq:contains_R_ball} that the image set~$\phi_\gamma(K)$ meets its image by the conjugate map $\phi_\gamma \circ \gamma \circ \phi_\gamma^{-1}$. We conclude that the original set~$K$ meets its image by the original map~$\gamma$.
\end{proof}

\bibliographystyle{alpha}
\bibliography{/home/ilia/Documents/Travaux_mathematiques/mybibliography.bib}

\noindent {\scshape Yale University Mathematics Department, PO Box 208283, New Haven, CT 06520-8283, USA} \\
\noindent {\itshape E-mail address:} \url{ilia.smilga@normalesup.org}
\end{document}